\documentclass[11pt]{article}

\usepackage{amsmath,amssymb,amsthm,amscd,amsfonts}
\usepackage{mathtools}
\usepackage{epsfig,pinlabel,paralist}
\usepackage[vmargin=1in, hmargin=1.25in]{geometry}
\usepackage[font=small,format=plain,labelfont=bf,up,textfont=it,up]{caption}
\usepackage{type1cm}
\usepackage{calc}
\usepackage{enumerate}
\usepackage{bm}
\usepackage[all,cmtip]{xy}
\usepackage{eucal}
\usepackage{paralist}
\usepackage{lmodern}
\usepackage[T1]{fontenc}
\usepackage{etoolbox}

\usepackage[bookmarks, bookmarksdepth=2, colorlinks=true, linkcolor=blue, citecolor=blue, urlcolor=blue]{hyperref}

\apptocmd{\thebibliography}{\raggedright}{}{}

\numberwithin{equation}{section}

\theoremstyle{plain}
\newtheorem{theorem}{Theorem}[section]

\newtheorem{lemma}[theorem]{Lemma}

\theoremstyle{definition}

\newtheorem{defn}[theorem]{Definition}

\theoremstyle{remark}
\newtheorem{rmk}[theorem]{Remark}
\newenvironment{remark}[1][]{\begin{rmk}[#1] \pushQED{}}{\popQED \end{rmk}}

\newtheorem{eg}[theorem]{Example}

\DeclareMathOperator{\Mod}{Mod}

\newcommand\R{\ensuremath{\mathbb{R}}}

\newcommand\Z{\ensuremath{\mathbb{Z}}}

\newcommand\Set[2]{\ensuremath{\left\{\text{#1 $|$ #2}\right\}}}

\newcommand\Figure[4]{
\begin{figure}[t]
\centering
\centerline{\psfig{file=#2,scale=#4}}
\caption{#3}
\label{#1}
\end{figure}}

\DeclareMathOperator{\cl}{cl}
\DeclareMathOperator{\scl}{scl}

\DeclareMathOperator{\oc}{oc}
\newcommand{\p}[1]{{\bf #1.}}

\title{\vspace{-40pt}Surface groups, infinite generating sets, and stable commutator length\vspace{-15pt}}
\author{Dan Margalit\thanks{Supported in part by NSF grants DMS-1510556 and DMS-1811941} \and Andrew Putman\thanks{Supported in part by NSF grants DMS-1737434 and DMS-1811210}}
\date{}

\begin{document}

\vspace{-10pt}
\maketitle

\vspace{-18pt}
\begin{abstract}
\noindent
We give a new proof of a theorem of D.\ Calegari that says that the Cayley graph of a surface group with respect to any generating set lying in finitely many mapping class group orbits has infinite diameter.  This applies, for instance, to
the generating set consisting of all simple closed curves.
\end{abstract}

\section{Introduction}
\label{section:introduction}

The objective of this note is to study the geometry of the Cayley graph of a surface group with respect to certain  infinite generating sets.  We give an alternate proof of a theorem of D.\ Calegari that addresses one
of the first geometric questions one might ask about these Cayley graphs: whether or not their
diameter is infinite.  

Let $\Sigma$ be a closed oriented surface of genus at least $2$ and let $\ast \in \Sigma$ be a base point.  We abbreviate $\pi_1(\Sigma,\ast)$ by $\pi_1(\Sigma)$.  The relative mapping class group $\Mod(\Sigma,\ast)$ is the group of isotopy classes of orientation-preserving diffeomorphisms of $\Sigma$ that fix $\ast$.
The group $\Mod(\Sigma,\ast)$ acts on $\pi_1(\Sigma)$ by automorphisms, and Calegari's theorem is as follows. 

\begin{theorem}[D.\ Calegari, \cite{Calegari}]
\label{theorem:main}
Let $\Sigma$ be a closed oriented surface of genus at least $2$ and let 
$S \subset \pi_1(\Sigma)$ be a generating
set contained in finitely many $\Mod(\Sigma,\ast)$-orbits.  Then the Cayley graph of
$\pi_1(\Sigma)$ with respect to $S$ has infinite diameter.
\end{theorem}

\vspace{-\topsep}
This theorem applies, for instance, to the set $S$ of all simple closed curves in $\pi_1(\Sigma)$, which satisfies the hypothesis of Theorem \ref{theorem:main} by the so-called 
change of coordinates principle; see \cite[\S 1.3]{FarbMargalitPrimer}.  Farb had asked if Theorem \ref{theorem:main} was true in this special case; see \cite{Calegari}.  In fact, the general case follows from this one; see \S \ref{section:remark} below.

Both Calegari's proof and ours make use of the theory of quasi-morphisms.  However, while he constructs appropriate quasi-morphisms directly using the hyperbolic geometry of $\Sigma$, we instead embed $\pi_1(\Sigma)$ into $\Mod(\Sigma,\ast)$ as the point-pushing subgroup and use the theory of stable commutator length on $\Mod(\Sigma,\ast)$.

The idea of using quasimorphisms on $\Mod(\Sigma,\ast)$ to study its point-pushing subgroup already
appeared in work of Brandenbursky--Marcinkowski \cite{BM}, who proved a stronger version of
Theorem \ref{theorem:main}.  Their theorem says that many elements of $\pi_1(\Sigma)$ are undistorted with respect to the word metrics appearing in Theorem \ref{theorem:main}.  We handle the details of the proof differently, and
in particular introduce a novel invariant of immersed closed curves in a surface that is
of independent interest; see \S \ref{section:nonreverse}.

Beyond Theorem \ref{theorem:main}, little is known about the geometry and group theory of $\pi_1(\Sigma)$ with
respect to infinite generating sets.  One exception is a theorem of the second author \cite{PutmanInfinite} that 
gives a complete set of relations between the set of simple closed curves in $\pi_1(\Sigma)$.  A natural next question is whether the Cayley graph of $\pi_1(\Sigma)$ with respect to the generating set consisting of all simple closed curves is $\delta$-hyperbolic.

\p{Outline} In Section~\ref{section:nonreverse} we construct an element $x$ of $\pi_1(\Sigma)$ such that no element of $\Mod(\Sigma,\ast)$ takes $x$ to $x^{-1}$.  There is a well-known inclusion of $\pi_1(\Sigma)$ into $\Mod(\Sigma,\ast)$ called the point-pushing map. Under this inclusion, $x$ corresponds to a pseudo-Anosov mapping class.  
In Section~\ref{section:mainproof} we apply a result of Calegari--Fujiwara to conclude that the stable commutator length of $x$ (as an element of $\Mod(\Sigma,\ast)$) is positive.  On the other hand, we argue that if the Cayley graph in the statement of Theorem~\ref{theorem:main} had finite diameter, then the stable commutator length function on $\pi_1(\Sigma) \leqslant \Mod(\Sigma,\ast)$ would be the zero function.  We close the paper in \S \ref{section:remark} with a discussion of how to derive Theorem~\ref{theorem:main} from the special case where the generating set $S$ consists of all simple closed curves.


\section{An irreversible curve}
\label{section:nonreverse}

The goal of this section is to prove the following lemma.  In its statement, we denote the oriented closed curve
obtained by reversing the orientation of an oriented closed curve $\gamma$ by $\bar \gamma$.
The (unbased) homotopy class of an oriented closed curve $\gamma$ is denoted $[\gamma]$.  
Finally, $\Mod(\Sigma)$ is the mapping class group of
a surface $\Sigma$, i.e.\ the group of isotopy classes of orientation-preserving diffeomorphisms of $\Sigma$.

\begin{lemma}
\label{lemma:noreverse}
Let $\Sigma$ be a closed oriented surface of genus at least $2$ and let $\alpha$ be the oriented closed curve in $\Sigma$ shown in Figure~\ref{figure:fillingcurve}.  Then $[\alpha]$ and $\left[\bar \alpha\right]$ do not lie in the same $\Mod(\Sigma)$-orbit.  
\end{lemma}

In the statement of the lemma and throughout the paper, when we refer to the curve $\alpha$ we are of course referring to any of the infinitely many curves indicated in Figure~\ref{figure:fillingcurve}; there is one such curve for each closed orientable surface of genus at least 2.

\vspace{-\topsep}
Our strategy for proving Lemma~\ref{lemma:noreverse} is to introduce an invariant for homotopy classes of oriented immersed closed curves in a surface (the homotopy classes of left-turning one-cornered simple subcurves) and to show that $\alpha$ and $\bar \alpha$ have distinct invariants.

\p{One-cornered subcurves}  Let $\gamma$ be an oriented immersed closed curve in a surface with no self-tangencies.  A {\em one-cornered simple subcurve} of $\gamma$ is an oriented closed curve $\delta$ with the following two properties:
\vspace{-\topsep}
\begin{compactitem}
\item $\delta$ is obtained by starting at a self-intersection point $p$ of $\gamma$ and following
$\gamma$ until it returns for the first time to $p$, and 
\item $\delta$ is freely homotopic to a simple closed curve.
\end{compactitem}
\vspace{-\topsep}
The self-intersection point $p$ at which $\delta$ begins and ends is the {\em corner} of $\delta$.  

We can also specify a direction of turning for $\delta$.  Let $\vec{v}$
and $\vec{w}$ be the tangent vectors to $\delta$ at its starting and ending
points, so $\vec{v}$ and $\vec{w}$ form a basis for the tangent space to the
surface at the corner $p$ of $\delta$.  We will say that $\delta$ is {\em right-turning}
if $\left(\vec{v},\vec{w}\right)$ is a positively oriented basis for the tangent space
and that $\delta$ is {\em left-turning} otherwise.

\Figure{figure:fillingcurve}{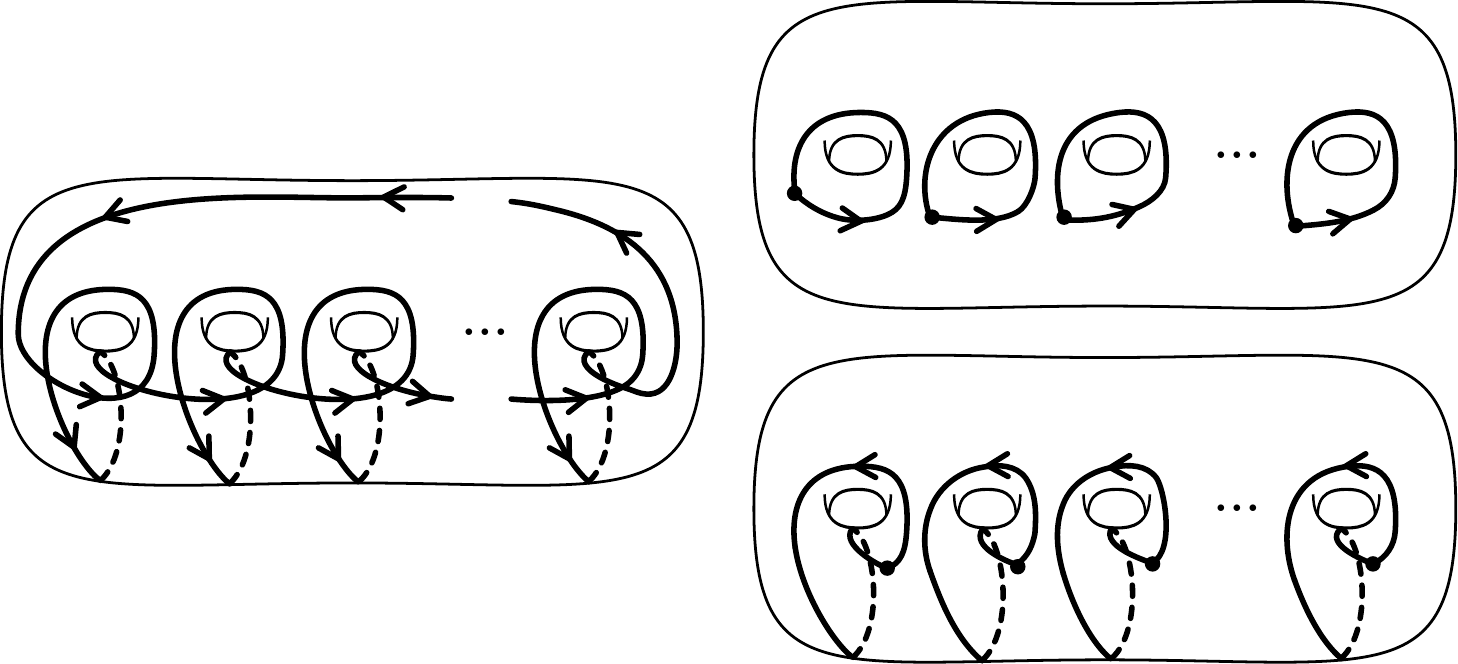}
{On the left is the oriented curve $\alpha$.  
On the right are the simple one-cornered subcurves of $\alpha$.  Each of them is
left-turning.}{90}

All of the one-cornered simple subcurves of the immersed closed curve $\alpha$
in the left-hand side of Figure~\ref{figure:fillingcurve} are shown in the right-hand side of
Figure~\ref{figure:fillingcurve}.  All of these one-cornered simple subcurves 
are left-turning. 

\p{The homotopy invariants} Let $\gamma$ be an oriented immersed closed curve with no self-tangencies.  We now use the notions of left- and right-turning one-cornered simple subcurves of  $\gamma$ in order to define an invariant of the homotopy class of $\gamma$.  Let
\[
\oc_L(\gamma) = \Set{$\left[\delta\right]$}{$\delta$ is a left-turning one-cornered simple subcurve of $\gamma$}
\]
and
\[
\oc_R(\gamma) = \Set{$\left[\delta\right]$}{$\delta$ is a right-turning one-cornered simple subcurve of $\gamma$}.
\]
For the curve $\alpha$ shown in Figure~\ref{figure:fillingcurve}, the set $\oc_L(\alpha)$ is the (nonempty) set of homotopy classes of curves shown in the right-hand side of Figure~\ref{figure:fillingcurve} and the set $\oc_R(\alpha)$ is empty.

The next lemma says that, at least under favorable circumstances, the sets $\oc_L(\gamma)$ and $\oc_R(\gamma)$ are invariant under homotopies of $\gamma$.  In its statement, we say that $\gamma$ is in {\em minimal position} if it is an immersed closed curve with no triple points and if it has the minimal possible number of self-intersections among all such curves that are homotopic to it.

\begin{lemma}
\label{lemma:invariance}
Let $\gamma$ and $\gamma'$ be homotopic oriented immersed closed curves
in minimal position.  Then $\oc_L(\gamma) = \oc_L(\gamma')$ and
$\oc_R(\gamma) = \oc_R(\gamma')$.
\end{lemma}
\vspace{-\topsep}
\begin{proof}
Since $\gamma$ and $\gamma'$ are in minimal position, a theorem of Hass--Scott \cite{HassScottHomotopy} applies to show that $\gamma$ and $\gamma'$ can be connected by a sequence of ambient isotopies and the ``Reidemeister III moves'' shown in the left-hand side of Figure~\ref{figure:reidemeister} (we remark that Hass--Scott do not state their result in this form, but
Paterson \cite[Theorem 1.2]{Paterson} gives an alternate proof of their result and states it as we have).  It is thus enough to deal with the case where $\gamma$
and $\gamma'$ differ by a single Reidemeister III move.  This is a straightforward check.
\end{proof}

\Figure{figure:reidemeister}{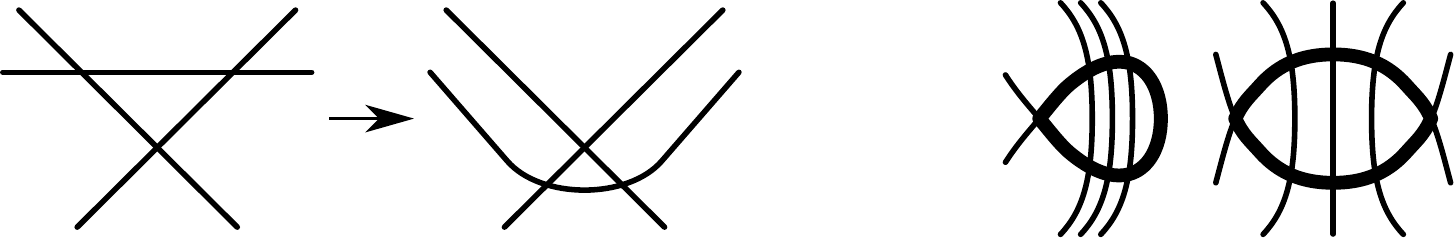}{On the left is the Reidemeister III move.  On the right are a monogon
and a bigon (both nonsingular).}{90}

\p{Checking minimal position} Lemma~\ref{lemma:invariance} is only as useful as our ability to check if a curve is in minimal position.  We recall here a theorem of Hass--Scott that gives a checkable criterion for a curve to be in minimal position, and use this to conclude that the curve $\alpha$ from Figure~\ref{figure:fillingcurve} is in minimal position (Lemma~\ref{lemma:minpos} below).  The statement of the criterion requires two definitions.  Let $\gamma$ be an immersed closed curve in $\Sigma$ with no self-tangencies. 

A \emph{singular monogon} in $\gamma$ is a subarc $\delta$ of $\gamma$ with the following two properties:
\vspace{-\topsep}
\begin{compactitem}
\item $\delta$ starts and ends at the same self-intersection point of $\gamma$ and
\item $\delta$ is nullhomotopic by a homotopy fixing its endpoints.
\end{compactitem}
\vspace{-\topsep}
A prototypical example of this is a nonsingular monogon, i.e.\ one
where $\delta$ is a simple closed curve in the surface that bounds a disc; see the right-hand side of Figure~\ref{figure:reidemeister}.

A \emph{singular bigon} in $\gamma$ is a pair of distinct subarcs $\delta$ and $\delta'$ of $\gamma$ with the
following three properties for some choice of orientation of $\delta$ and $\delta'$:
\vspace{-\topsep}
\begin{compactitem}
\item $\delta$ and $\delta'$ start at the same self-intersection point of $\gamma$, 
\item $\delta$ and $\delta'$ end at the same self-intersection point of $\gamma$, and
\item $\delta$ and $\delta'$ are homotopic by a homotopy that fixes their starting and ending points.
\end{compactitem}
\vspace{-\topsep}
A prototypical example of this is a nonsingular bigon, i.e.\ one where $\delta \cup \delta'$ is a simple closed curve in the surface that bounds a disc; see the right-hand side of Figure~\ref{figure:reidemeister}.

Hass--Scott proved that an immersed closed curve in a surface with no triple points and no self-tangencies 
is in minimal position if and only if it does not have any singular monogons or singular bigons; see Paterson \cite{Paterson} for an alternate proof.

The curve $\alpha$ from Figure~\ref{figure:fillingcurve} has no triple points and does not have any singular monogons or singular bigons.  We thus have the following lemma.

\begin{lemma}\label{lemma:minpos}
Let $\Sigma$ be a closed oriented surface of genus at least $2$ and let $\alpha$ be the oriented closed curve in $\Sigma$ shown in Figure~\ref{figure:fillingcurve}.  Then $\alpha$ is in minimal position.
\end{lemma}

\p{Finishing the proof} We now apply
Lemmas~\ref{lemma:invariance} and~\ref{lemma:minpos} to prove Lemma~\ref{lemma:noreverse}, which states that $[\alpha]$ and $[\bar \alpha]$ do not lie in the same $\Mod(\Sigma)$-orbit.

\begin{proof}[Proof of Lemma~\ref{lemma:noreverse}]
As in the statement, let $\alpha$ be the oriented immersed closed curve in the surface $\Sigma$ depicted in the left-hand side
of Figure~\ref{figure:fillingcurve}.  Let $\phi$ be an orientation-preserving diffeomorphism
of $\Sigma$ and $[\phi] \in \Mod(\Sigma)$ its mapping class.  We must prove that $[\phi(\alpha)] \neq [\bar \alpha]$.  

As discussed above, $\oc_L(\alpha) \neq \emptyset$ and $\oc_R(\alpha) = \emptyset$.  We thus have that
\[
\oc_L(\bar \alpha) =\oc_R(\alpha) = \emptyset
\]
and
\[
\oc_L\left(\phi\left(\alpha\right)\right) = [\phi] \left(\oc_L\left(\alpha\right) \right) \neq \emptyset.
\]
In particular, $\oc_L(\bar \alpha) \neq \oc_L\left(\phi\left(\alpha\right)\right)$.

By Lemma~\ref{lemma:minpos}, the curve $\alpha$ is in minimal position.
It follows that $\bar \alpha$ and $\phi(\alpha)$ are also in minimal position.  
By Lemma~\ref{lemma:invariance}, we have
$\left[\phi(\alpha)\right] \neq [\bar \alpha]$, as desired.
\end{proof}


\section{The proof of Theorem~\ref{theorem:main}}
\label{section:mainproof}

In this section, we prove Theorem~\ref{theorem:main}.  The proof involves the theory of stable commutator length, so we begin by recalling some basic facts about this.

\p{Stable commutator length} Let $G$ be a group.  For $g \in [G,G]$, define the \emph{commutator length} 
$\cl(g)$ 
to be the minimal number of commutators needed to express $g$.  The function
\[\cl\colon [G,G] \rightarrow \Z_{\geq 0}\]
is subadditive in the sense that
\[
\cl(gh) \leq \cl(g) + \cl(h)
\]
for all $g,h \in [G,G]$.
The \emph{stable commutator length} of $g \in [G,G]$ is the real number
\[\scl(g) = \lim_{n \to \infty} \frac{1}{n} \cl(g^n).\]
The subadditivity implies that this limit exists.  See Calegari's book \cite{CalegariBook} for
a survey of stable commutator length.

\p{A criterion for a Cayley graph to have infinite diameter} Theorem~\ref{theorem:main} states that certain Cayley graphs for $\pi_1(\Sigma)$ have infinite diameter.  Since
$[\pi_1(\Sigma),\pi_1(\Sigma)] \neq \pi_1(\Sigma)$, we cannot apply the theory of
stable commutator length directly to prove this theorem.  Instead, we will apply
the following lemma.

\begin{lemma}
\label{lemma:cayley}
Let $G$ be a group and let $\scl\colon [G,G] \to \R$ be its stable commutator length
function.  Let $H$ be a normal subgroup of $G$ such that $H < [G,G]$.  The group $G$
acts on $H$ by conjugation; let $S$ be a generating set for $H$ that is contained
in finitely many $G$-orbits.  Assume that the Cayley graph of $H$ with respect
to $S$ has finite diameter.  Then $\scl(h) = 0$ for all $h \in H$.
\end{lemma}

We emphasize that in the statement of Lemma \ref{lemma:cayley}  the expression $\scl(h)$ refers
to the stable commutator length of $h$ as an element of $G$.

\begin{proof}[Proof of Lemma \ref{lemma:cayley}]
Since the Cayley graph of $H$ with respect to $S$ has finite diameter, there exists
some $C \geq 0$ such that every element of $H$ can be written as a product of
at most $C$ elements of $S \cup S^{-1}$.  Let $\cl\colon [G,G] \rightarrow \Z$ be
the commutator length function and define
\[D = \sup\Set{$\cl(s)$}{$s \in S \cup S^{-1}$}.\]
Since the commutator length function on $[G,G]$ is invariant under the
conjugation action of $G$ and $S$ is contained in finitely many $G$-orbits, 
we have $D < \infty$.  For all $h \in H$, we then have
\[
\cl(h) \leq C D
\]
and thus
\[
\scl(h) = \lim_{n \mapsto \infty} \frac{1}{n} \cl(h^n) \leq \lim_{n \mapsto \infty} \frac{1}{n} C D = 0.\qedhere
\]
\end{proof}

\p{Surface subgroups of mapping class groups}
Let $\Sigma$ be a closed oriented surface of genus at least $2$.  The group
$\pi_1(\Sigma)$ can be embedded into $\Mod(\Sigma,\ast)$ as the {\em point-pushing
subgroup}; see \cite[\S 4.2]{FarbMargalitPrimer}.  This point-pushing subgroup
is the kernel of the natural surjective map $\Mod(\Sigma,\ast) \rightarrow \Mod(\Sigma)$,
so we have a short exact sequence
\[
1 \to \pi_1(\Sigma) \to \Mod(\Sigma,\ast) \to \Mod(\Sigma) \to 1
\]
called the Birman exact sequence.  

\begin{remark}
There is one annoying technical issue here.  In $\pi_1(\Sigma)$ we multiply paths
from left to right, but in $\Mod(\Sigma,\ast)$ we compose mapping classes from right
to left.  The naive inclusion $\pi_1(\Sigma) \rightarrow \Mod(\Sigma,\ast)$
is thus an anti-homomorphism rather than a homomorphism.  To fix this, we compose
the inclusion $\pi_1(\Sigma) \rightarrow \Mod(\Sigma,\ast)$ with the inversion
map on $\pi_1(\Sigma)$, so $x \in \pi_1(\Sigma)$ corresponds to the point-pushing
mapping class that pushes $\ast$ around $x^{-1}$.
\end{remark}

We will need the following two properties of the point-pushing subgroup:
\vspace{-\topsep}
\begin{compactitem}
\item the conjugation action of $\Mod(\Sigma,\ast)$ on the normal subgroup $\pi_1(\Sigma)$ is the natural one induced by the action of diffeomorphisms on $\Sigma$ and 
\item $\pi_1(\Sigma)$ lies in the commutator subgroup of $\Mod(\Sigma,\ast)$.
\end{compactitem}
\vspace{-\topsep}
The second property is vacuous when the genus of $\Sigma$ is at least $3$ since in these cases $\Mod(\Sigma,\ast)$ is perfect.  For the case of genus $2$, we observe that $\pi_1(\Sigma)$ is generated by nonseparating simple loops and thus the point-pushing subgroup is generated by products $T_a T_b^{-1}$, where $T_a$ and $T_b$ are Dehn twists about nonseparating curves in $\Sigma$; see \cite[\S 4.2]{FarbMargalitPrimer}.  Since all Dehn twists about nonseparating curves are conjugate in $\Mod(\Sigma,\ast)$, it follows that $T_a T_b^{-1}$ (and hence all of the point pushing subgroup) lies in the commutator subgroup of $\Mod(\Sigma,\ast)$.

\p{Finishing the proof} We are now ready to prove our main theorem.  The proof will bring to bear the Nielsen--Thurston classification, which classifies elements of the mapping class group as either periodic, reducible, or pseudo-Anosov; see \cite[\S 13.3]{FarbMargalitPrimer} for details and background.

\begin{proof}[Proof of Theorem~\ref{theorem:main}]
By Lemma~\ref{lemma:cayley}, it is enough to exhibit a single element of $\pi_1(\Sigma)$ whose $\Mod(\Sigma,\ast)$-stable commutator length is nonzero.
Let $x \in \pi_1(\Sigma)$ be any element freely homotopic to the curve $\alpha$ from Lemma~\ref{lemma:noreverse}.  The curve $x$ fills $\Sigma$, meaning that it has a minimal position representative where all of the complementary regions are disks.  In fact we may choose the base point $\ast$ and the element $x$ such that the minimal position representative is $\alpha$.  By a theorem of Kra \cite{Kra}, the associated element of $\Mod(\Sigma,\ast)$ is pseudo-Anosov (see also \cite[\S 14.1.4]{FarbMargalitPrimer}).

By Lemma~\ref{lemma:noreverse}, there is no $f \in \Mod(\Sigma,\ast)$ such that 
$f_{\ast}(x) = x^{-1}$.  Since surface groups have unique roots (see, e.g.\ 
\cite[Problem 1.20 \& Theorem 3.11]{ClayRolfsen}), it then follows that for each $n \geq 1$ there does not exist an $f \in \Mod(\Sigma,\ast)$ with $f_{\ast}(x^n) = x^{-n}$.
In other words, no positive power of the mapping class associated to $x$ is
conjugate in $\Mod(\Sigma,\ast)$ to its inverse.
A theorem of Calegari--Fujiwara \cite[Theorem C]{CalegariFujiwara} states that if a pseudo-Anosov element of $\Mod(\Sigma,\ast)$ has no positive power that is conjugate to its inverse, then its stable commutator length is nonzero.  In particular the stable commutator length of $x \in \Mod(\Sigma,\ast)$ is nonzero, as desired.  
\end{proof}

\begin{remark}
The theorem of Calegari--Fujiwara used in the proof above was later generalized
by Bestvina--Bromberg--Fujiwara \cite{BestvinaBrombergFujiwara} 
to give a complete characterization of which
mapping classes have positive stable commutator length.
\end{remark}

\section{Deriving the general case from the special case}
\label{section:remark}

Neither our proof nor Calegari's original proof of Theorem \ref{theorem:main} appear to simplify
in the case where the generating set $S$ is the set of all simple closed curves in $\pi_1(\Sigma)$.
However, it is interesting to observe that the general case follows from this one.  Indeed, assume
that the Cayley graph of $\pi_1(\Sigma)$ with respect to $S$ has infinite diameter, and consider
any generating set $S' \subset \pi_1(\Sigma)$ that is contained in finitely many $\Mod(\Sigma,\ast)$-orbits.
We will show how to deduce that the Cayley graph of $\pi_1(\Sigma)$ with respect to $S'$ has infinite diameter.
Enlarging $S'$ can only shrink the diameter of the Cayley graph, so without loss of generality we
can assume that $S' = \Mod(\Sigma,\ast) \cdot S_0'$ for some finite generating set $S_0'$ for $\pi_1(\Sigma)$.  For
later use, let $S_0 \subset S$ also be a finite generating set for $\pi_1(\Sigma)$ such that $S = \Mod(\Sigma,\ast) \cdot S_0$.

It is enough now to prove that the Cayley graphs of $\pi_1(\Sigma)$ with respect to $S$ and
$S'$ are quasi-isometric.  
Letting $\|\cdot\|_S$ and $\|\cdot\|_{S'}$ be the associated word norms on $\pi_1(\Sigma)$, this is equivalent to showing
that there exists some $C,C' \geq 1$ such that
\[\|w\|_{S'} \leq C \|w\|_{S} \quad \text{and} \quad \|w\|_{S} \leq C' \|w\|_{S'}\]
for all $w \in \pi_1(\Sigma)$.  The proofs of the existence of $C$ and $C'$ are the same, except with the roles of $\{S,S_0\}$ and $\{S',S_0'\}$ reversed; we will write the argument for $C$.

Since $S_0$ is a finite generating set for $\pi_1(\Sigma)$, there exists some $C \geq 1$ such that
$\|s\|_{S'} \leq C$ for all $s \in S_0$.  For $\phi \in \Mod(\Sigma)$ and $s \in S_0$, we thus have
\[\|\phi(s)\|_{S'} = \|s\|_{\phi^{-1}(S')} = \|s\|_{S'} \leq C;\]
here we are using the fact that $S'$ is $\Mod(\Sigma)$-invariant.
Since $S = \Mod(\Sigma) \cdot S_0$, this implies that $\|s\|_{S'} \leq C$ for all $s \in S$.  From
this, we see that $\|w\|_{S'} \leq C \|w\|_{S}$ for all $w \in \pi_1(\Sigma)$, as desired.

\begin{footnotesize}
\noindent
\begin{tabular*}{\linewidth}[t]{@{}p{\widthof{Department of Mathematics}+0.5in}@{}p{\linewidth - \widthof{Department of Mathematics} - 0.5in}@{}}
{\raggedright
Dan Margalit\\
Georgia Institute of Technology\\
School of Mathematics\\
686 Cherry St\\ 
Atlanta, GA 30306\\
{\tt margalit@math.gatech.edu}}
&
{\raggedright
Andrew Putman\\
Department of Mathematics\\
University of Notre Dame \\
255 Hurley Hall\\
Notre Dame, IN 46556\\
{\tt andyp@nd.edu}}
\end{tabular*}\hfill
\end{footnotesize}

\end{document}